\documentclass[10pt]{amsart}

 \usepackage{hyperref}

\hypersetup{colorlinks=true, urlcolor=blue, citecolor=blue, linkcolor=blue}

\usepackage{url}

\usepackage{amsmath,amsfonts, latexsym,graphicx, footmisc, amssymb, mathtools}

\newcommand{\U}{{\mathcal U}}
\newcommand{\0}{{\mathbf 0}}
\newcommand{\C}{{\mathbb C}}
\newcommand{\Z}{{\mathbb Z}}

\newcommand{\cL}{{\mathbb L}}
\newcommand{\D}{{\mathbb D}}

\newcommand{\hyp}{{\mathbb H}}

\newcommand{\Adot}{\mathbf A^\bullet}
\newcommand{\Bdot}{\mathbf B^\bullet}

\newcommand{\Idot}{\mathbf I^\bullet}

\newcommand{\Pdot}{\mathbf P^\bullet}

\newcommand{\Ndot}{\mathbf N^\bullet}

\newcommand{\coker}{{\operatorname{coker}}}

\newtheorem{defn0}{Definition}[section]
\newtheorem{prop0}[defn0]{Proposition}
\newtheorem{conj0}[defn0]{Conjecture}
\newtheorem{thm0}[defn0]{Theorem}
\newtheorem{lem0}[defn0]{Lemma}
\newtheorem{corollary0}[defn0]{Corollary}
\newtheorem{example0}[defn0]{Example}
\newtheorem{remark0}[defn0]{Remark}
\newtheorem{question0}[defn0]{Question}

\newenvironment{defn}{\begin{defn0}}{\end{defn0}}
\newenvironment{prop}{\begin{prop0}}{\end{prop0}}

\newenvironment{thm}{\begin{thm0}}{\end{thm0}}
\newenvironment{lem}{\begin{lem0}}{\end{lem0}}
\newenvironment{cor}{\begin{corollary0}}{\end{corollary0}}

\newenvironment{rem}{\begin{remark0}\rm}{\end{remark0}}

\newcommand{\defref}[1]{Definition~\ref{#1}}
\newcommand{\propref}[1]{Proposition~\ref{#1}}
\newcommand{\thmref}[1]{Theorem~\ref{#1}}
\newcommand{\lemref}[1]{Lemma~\ref{#1}}
\newcommand{\corref}[1]{Corollary~\ref{#1}}

\newcommand{\secref}[1]{Section~\ref{#1}}

\title[IC and Perverse Eigenspaces of the Monodromy]{Intersection Cohomology and Perverse Eigenspaces of the Monodromy}


\author{David B. Massey}

\date{}

\begin{document}

\begin{abstract} We describe the relationship between intersection cohomology
with twisted coefficients and the perverse sheaves which play the role of the eigenspaces for the Milnor monodromy of an affine hypersurface.
\end{abstract}

\maketitle




\section{Introduction} The Monodromy Theorem (\cite{clemens}, \cite{sgavii1}, \cite{landman}, \cite{lemono1}) tells us that the classical Milnor monodromy for an affine hypersurface is quasi-unipotent; over $\C$, this means that the eigenvalues of the monodromy of the Milnor fiber at each point in the hypersurface are  roots of unity. We wish to look at this in the category of perverse sheaves and see what it has to do with intersection cohomology, but first we need to describe our set-up; general references for this background are \cite{BBD}, \cite{inthom2}, \cite{kashsch},  \cite{dimcasheaves}, \cite{levan}, and \cite{schurbook}.

\bigskip

Suppose that $\U$ is a non-empty open subset of $\C^{n+1}$, where $n\geq 1$, and that $f:\U\rightarrow\C$ is a reduced, nowhere locally constant, complex analytic function. Then $X:=V(f)=f^{-1}(0)$ is a hypersurface in $\U$ of pure dimension $n$. Furthermore, since $f$ is reduced, the singular set $\Sigma$ of $X$ is equal to the intersection of the hypersurface with the critical locus, $\Sigma f$, of $f$. 

If we use $\C$ for our base ring for perverse sheaves, then the shifted constant sheaf $\C^\bullet_\U[n+1]$ is perverse on $\U$, as is the shifted constant sheaf $\C^\bullet_X[n]$ on $X$. Furthermore, on $X$, there are the perverse sheaves of nearby and vanishing cycles of $\C^\bullet_\U[n+1]$ along $f$, denoted, respectively, as $\psi_f[-1]\C^\bullet_\U[n+1]$ and $\phi_f[-1]\C^\bullet_\U[n+1]$. The stalk cohomology of these complexes at a point $x\in X$ yield the ordinary cohomology and reduced cohomology of the Milnor fiber $F_{f,x}$ of $f$ at $x$ with a shift.

There are monodromy automorphisms $T_f$ and $\widetilde T_f$ on the nearby and vanishing cycles, respectively. On stalk cohomology, these yield the ordinary Milnor monodromy automorphisms. For each root of unity $\xi$ (actually, for any complex number), we may consider the perverse eigenspaces of these automorphisms, namely $\ker\{\xi\operatorname{id}-T_f\}$ and $\ker\{\xi\operatorname{id}-\widetilde T_f\}$. It is important to note that the stalk cohomology of these perverse kernels need {\bf not} be isomorphic to the eigenspaces of the monodromy on the stalk cohomology.

When $\xi=1$, we can relate $\ker\{\xi\operatorname{id}-\widetilde T_f\}$ to the intersection cohomology $\Idot_X$, with constant coefficients, on $X$; in fact, we can do this over $\Z$, not just over $\C$. We show:
\begin{thm} \textnormal{(\propref{prop:easyprop} and \thmref{thm:1theorem})}\label{intro:thm1} Using $\Z$ as our base ring, there is an isomorphism $\Z_X[n]\cong \ker\{\operatorname{id}- T_f\}$, and the kernel of the canonical perverse surjection $\Z_X[n]\twoheadrightarrow \Idot_X$ is isomorphic to
$\ker\{\operatorname{id}-\widetilde T_f\}$. 

The analogous statements over $\C$ are also true.
\end{thm}

For $\xi\neq1$, we must use intersection cohomology $\Idot_\U(\xi)$ on $\U$ with a twist by $\xi$ around the hypersurface (we shall make this precise later). Letting $j:X\hookrightarrow\U$ be the inclusion, we show:

\begin{thm}\label{intro:thm2} \textnormal{(\thmref{thm:main2})} Using $\C$ as our base ring, and supposing the $\xi\neq 0$ or $1$, there is an isomorphism of perverse sheaves 
$$j^*[-1]\Idot_\U(\xi)\cong\ker\{\xi^{-1}\operatorname{id}- T_f\}\cong \ker\{\xi^{-1}\operatorname{id}-\widetilde T_f\}.$$
\end{thm}

\smallskip

We begin with the case of $\xi=1$; it is, in fact, the complicated case. First, in \secref{sec:compcomp}, we look at some basic results about  the kernel of the canonical perverse surjection $\Z_X[n]\twoheadrightarrow \Idot_X$. Then we prove the main theorem in the $\xi=1$ case in \secref{sec:main1}. We prove the main theorem for $\xi\neq1$ in \secref{sec:mainnot1}.

\smallskip

In \secref{sec:appl}, we will give some down-to-Earth applications of these results, applications which do not, when possible, use the language of the derived category.

\smallskip

As a final comment in this introduction, we should mention that in the algebraic setting with field coefficients, our results can be obtained from the work of Beilinson in \cite{beilinsonglue} (see, also, \cite{reich}) and the work of M. Saito in \cite{mixedhodgemod}. However, we do care about $\Z$ coefficients and the analytic case, and the proofs by these other techniques are not simpler.

\section{The Comparison Complex}\label{sec:compcomp}

We continue with $\U$, $f$, and $X$ as in the introduction. We denote the respective inclusions as follows: $j:X\hookrightarrow\U$, $i:\U\backslash X\hookrightarrow\U$, $m:\Sigma\hookrightarrow X$, and $l:X\backslash\Sigma\hookrightarrow X$. Furthermore, we let $\hat m:=j\circ m$ be the inclusion of $\Sigma$ into $\U$. In this section, our base ring is $\Z$ (though all statements hold with base ring $\C$).

In this setting, there (at least) three canonical perverse sheaves on $X$: the shifted constant sheaf $\Z^\bullet_X[n]\cong j^*[-1]\Z^\bullet_\U[n+1]$, the Verdier dual sheaf $j^![1]\Z^\bullet_\U[n+1]$, and the intersection cohomology complex $\mathbf I^\bullet_X$ (with constant $\Z$ coefficients), which is the intermediate extension to all of $X$ of $\Z^\bullet_{X\backslash \Sigma}[n]$. In addition, there is a canonical surjection 
$$
\Z^\bullet_X[n]\overset{\tau_{{}_X}}{\twoheadrightarrow}\mathbf I^\bullet_X
$$
in the Abelian category $\operatorname{Perv}(X)$ of perverse sheaves (with middle perversity) on $X$ (the surjectivity follows from the fact that $\tau_{{}_X}$ is an isomorphism on $X\backslash\Sigma$ and that the intermediate extension $\Idot_X$ has no non-trivial quotients with support contained in $\Sigma$). There is also a dual injection of $\mathbf I^\bullet_X$ into $j^![1]\Z^\bullet_\U[n+1]$. 

\medskip

We shall focus on the surjection $\tau_{{}_X}: \Z^\bullet_X[n]\twoheadrightarrow\mathbf I^\bullet_X$; one can dualize the results there to obtain results for the dual injection.

\smallskip

Our fundamental definition is:

\begin{defn}\label{defn:compcomplex} The {\bf comparison complex}, $\Ndot_X$, on $X$ is the kernel (in $\operatorname{Perv}(X)$) of $\tau_{{}_X}$. Hence, by definition, the support of $\Ndot_X$ is contained in $\Sigma$ and there is a short exact sequence
$$
0\rightarrow \Ndot_X\rightarrow\Z^\bullet_X[n]\xrightarrow{\tau_{{}_X}}\mathbf I^\bullet_X\rightarrow0.
$$
\end{defn}

\medskip

\begin{rem}\label{rem:reduced}

While $\Ndot_X$ is interesting as a perverse sheaf, on the level of stalks, $\Ndot_X$ merely gives the shifted, reduced intersection cohomology of $X$. To be precise, the long exact sequence on stalk cohomology at $x\in X$ yields a short exact sequence
$$
0\rightarrow \Z\rightarrow H^{-n}(\Idot_X)_x\rightarrow H^{-n+1}(\Ndot_X)_x\rightarrow 0
$$
and, for all $k\neq -n+1$, isomorphisms
$$
H^{k-1}(\Idot_X)_x\cong H^{k}(\Ndot_X)_x.
$$
If we let $\widetilde{IH}$ denote reduced intersection cohomology, with topological indexing (as is used for intersection homology in \cite{inthom}), then we have
$$
H^{k}(\Ndot_X)_x\cong\widetilde{IH}^{n+k-1}(B^\circ_\epsilon(x)\cap X; \,\Z),
$$
where $B^\circ_\epsilon(x)$ denotes a small open ball, of radius $\epsilon$, centered at $x$.

\end{rem}

\begin{rem} Our interest in the comparison complex arose from considering parameterized hypersurfaces, i.e., hypersurfaces which have smooth normalizations. Our results in this case appear in our joint work with Brian Hepler in \cite{heplermassparam} (see also Hepler's  paper \cite{hepler}), where it was more natural to refer to $\Ndot_X$ as the {\it multiple-point complex}.

\end{rem}

\bigskip

We wish to recall some basic definitions and results on intersection cohomology, real links, and complex links.

\smallskip

Let $B_\epsilon(x)$ denote the closed ball of radius $\epsilon$, centered at $x$, in $\C^{n+1}$, so that its boundary $\partial B_\epsilon(x)$ is a ($2n+1$)-dimensional sphere. We continue to denote the open ball by $B^\circ_\epsilon(x)$.

The {\it real link} of $X$ at $x$ is (the homeomorphism-type of) 
$$K_{X,x}:=X\cap\partial B_\epsilon(x)$$ 
for sufficiently small $\epsilon>0$, which is homotopy-equivalent to $X\cap(B^\circ_\epsilon(x)\backslash\{x\})$. See, for instance, \cite{milnorsing} and \cite{stratmorse}.

\smallskip

The support and cosupport conditions tell us that, for all $x\in X$, the stalk cohomology of $\Idot_X$ has the following properties:
$$
H^k(\Idot_X)_x\cong\begin{cases}\hyp^k(K_{X,x}; \, \Idot_X), \textnormal{ if } k\leq -1; \\
0, \textnormal{ if } k\geq 0.
\end{cases}
$$

In particular, we have:

\begin{prop} If $x$ is an isolated singular point of $X$, then $\Ndot_X$ is  a perverse sheaf which has $x$ as an isolated point in its support, and so $H^k(\Ndot_X)_x=0$ if $k\neq 0$, and
$$
H^0(\Ndot_X)_x\cong \widetilde H^{n-1}(K_{X,x}; \, \Z).
$$
\end{prop}

\bigskip

The {\it complex link} of $X$ at $x$ is (the homeomorphism-type of) 
$$\cL_{X, x}:=B^\circ_\epsilon(x)\cap X\cap L^{-1}(\gamma),$$
where $L$ is a generic affine linear form such that $L(x)=0$ and $0<|\gamma|\ll \epsilon\ll 1$ (see, for instance, \cite{stratmorse} and \cite{levan}). As $X$ is a hypersurface (and so, a local complete intersection), $\cL_{X, x}$ has the homotopy-type of a bouquet of a finite number of ($n-1$)-dimensional spheres. The number of spheres in this bouquet equals $\operatorname{rank}\widetilde H^{n-1}(\cL_{X, x};\,\Z))$, which equals the rank of $H^0(\phi_L[-1]\Z^\bullet_X[n])_x$. Furthermore, it is known that this rank is equal to an intersection number,
$$
\operatorname{rank}H^0(\phi_L[-1]\Z^\bullet_X[n])_x=\big(\Gamma_{f, L}^1\cdot V(L)\big)_x,
$$
where $\Gamma_{f, L}^1$ is the relative polar curve of $f$ with respect to $L$; see Corollary 2.6 of \cite{gencalcvan} (though this was known earlier by Hamm, L\^e, Siersma, and Teissier). 

\smallskip

We can prove a non-trivial proposition merely using the defining short exact sequence in \defref{defn:compcomplex}, and so this result really follows just from the surjectivity of $\tau_{{}_X}$. Compare with Proposition 6.1.22 of \cite{dimcasheaves}.

\begin{prop}\label{prop:easycool}Suppose that $x$ is an isolated singular point of $X$. Then,
$$
\operatorname{rank} H^{n-1}(\cL_{X, x};\, \Z)-\operatorname{rank} H^{n-1}(K_{X, x};\, \Z)=\operatorname{rank} H^0(\phi_L[-1]\Idot_X)_x\geq 0,$$
where $L$ is a generic affine linear form such that $L(x)=0$.
\end{prop}
\begin{proof} Applying $\phi_L[-1]$ to the defining short exact sequence for $\Ndot_x$, we obtain the short exact sequence
$$
0\rightarrow \phi_L[-1]\Ndot_X\rightarrow\phi_L[-1]\Z^\bullet_X[n]\rightarrow\phi_L[-1]\mathbf I^\bullet_X\rightarrow0,
$$
in which each term has support contained in $\{x\}$ locally; this implies that the long exact sequence on stalk cohomology at $x$ yields a short exact sequence
$$
0\rightarrow H^0(\phi_L[-1]\Ndot_X)_x\rightarrow H^0(\phi_L[-1]\Z^\bullet_X[n])_x\rightarrow H^0(\phi_L[-1]\mathbf I^\bullet_X)_x\rightarrow0.
$$

However, since $x$ is an isolated singular point of $X$, $x$ is an isolated point in the support of $\Ndot_x$; hence, 
$$
H^0(\phi_L[-1]\Ndot_X)_x\cong H^0(\Ndot_X)_x\cong \widetilde H^{n-1}(K_{X,x}; \, \Z).
$$
Therefore, from the previous short exact sequence, we obtain
$$
\operatorname{rank} \widetilde H^{n-1}(\cL_{X, x};\, \Z)-\operatorname{rank} \widetilde H^{n-1}(K_{X, x};\, \Z)= $$
$$
\operatorname{rank} H^0(\phi_L[-1]\Z^\bullet_X[n])_x-\operatorname{rank}H^0(\phi_L[-1]\Ndot_X)_x=\operatorname{rank} H^0(\phi_L[-1]\Idot_X)_x.
$$

Finally, the difference in the ranks of the reduced cohomologies equals the difference without the reductions.
\end{proof}

\smallskip

\begin{rem} Suppose that $X$ is a pure-dimensional (i.e., connected) local complete intersection of dimension $d$. Then, the result of L\^e in \cite{levan} implies that $\Z_X^\bullet[d]$ is perverse and, again, the canonical map from $\Z_X^\bullet[d]$ to the intersection cohomology on $X$ is a surjection. If $X$ has a isolated singular point at $x$, then the result of \propref{prop:easycool} remains true if one replaces $n$ with $d$.
\end{rem}

\section{The Main Theorem for $\xi=1$}\label{sec:main1}

We continue with the notation from the previous section and the introduction, and we continue to use $\Z$ as our base ring.

\medskip

Recall that
$$T_f: \psi_f[-1]\Z^\bullet_\U[n+1]\xrightarrow{\cong}\psi_f[-1]\Z^\bullet_\U[n+1]$$ 
and 
$$\widetilde T_f: \phi_f[-1]\Z^\bullet_\U[n+1]\xrightarrow{\cong}\phi_f[-1]\Z^\bullet_\U[n+1]$$
denote the monodromy automorphisms on the nearby and vanishing cycles along $f$, respectively.

Throughout this section and the remainder of the paper, we will use the notation from the introduction. In addition, we let $s:=\dim\Sigma$ (or, when we focus on the germ at $x\in \Sigma$, we will let $s:=\dim_x\Sigma$). 

\smallskip

\smallskip 

\smallskip

Recall that there are two canonical short exacts sequences in $\operatorname{Perv}(X)$:
$$
0\rightarrow j^*[-1]\Z^\bullet_\U[n+1]\rightarrow\psi_f[-1]\Z^\bullet_\U[n+1]\xrightarrow{\operatorname{can}}\phi_f[-1]\Z^\bullet_\U[n+1]\rightarrow 0
$$
and
$$
0\rightarrow \phi_f[-1]\Z^\bullet_\U[n+1]\xrightarrow{\operatorname{var}}\psi_f[-1]\Z^\bullet_\U[n+1]\rightarrow j^![1]\Z^\bullet_\U[n+1]\rightarrow 0,
$$
where $\operatorname{var}\circ\operatorname{can}=\operatorname{id}-T_f$ and $\operatorname{can}\circ\operatorname{var}=\operatorname{id}-\widetilde T_f$. See, for instance, \cite{kashsch}, 8.6.7 and 8.6.8 (but be aware that the $\phi_f$ of Kashiwara and Schapira is our $\phi_f[-1]$), \cite{schurbook} 6.0.4, or Definition 4.2.4 and Remark 4.2.12 of \cite{dimcasheaves}.

\smallskip

From these short exact sequences and the fact that  $\operatorname{var}\circ\operatorname{can}=\operatorname{id}-T_f$, we immediately conclude:

\begin{prop}\label{prop:easyprop} There is an isomorphism of perverse sheaves
$$
\Z^\bullet_X[n]\cong\operatorname{ker}\big\{\operatorname{id}-T_f\big\}.
$$
Dually, there is an isomorphism of perverse sheaves
$$
 j^![1]\Z^\bullet_\U[n+1]\cong\operatorname{coker}\big\{\operatorname{id}-T_f\big\}.
$$

\end{prop}

\medskip

As we mentioned in the introduction, this proposition is very unsatisfying, since we get $\Z^\bullet_X[n]$ for the kernel, regardless of whether or not $1$ is an eigenvalue of the Milnor monodromy in some degree $>0$. However, as we shall see, the vanishing cycle analog, $\operatorname{ker}\big\{\operatorname{id}-\widetilde T_f\big\}$, is more interesting.

\smallskip

First, we need a lemma. We use ${}^{\mu}\hskip -0.02in H^k(\Adot)$ to denote the degree $k$ perverse cohomology of a complex (with middle perversity). See Section 10.3 of \cite{kashsch} or Sections 5.1 and 5.2 of \cite{dimcasheaves}.

\begin{lem}\label{lem:first} There is an isomorphism of perverse sheaves
$$
\Ndot_X\cong {}^{\mu}\hskip -0.02in H^0(m_!m^!\Z^\bullet_X[n]).
$$
\end{lem}
\begin{proof} The intersection cohomology complex $ \Idot_X$ on $X$ is the intermediate extension of the shifted constant on $X\backslash\Sigma$, i.e., the image (in $\operatorname{Perv}(X)$) of the canonical morphism
$$
{}^{\mu}\hskip -0.02in H^0(l_!\Z^\bullet_{X\backslash\Sigma}[n])\xrightarrow{\alpha}{}^{\mu}\hskip -0.02in H^0(l_*\Z^\bullet_{X\backslash\Sigma}[n]).
$$
The morphism $\alpha$ factors through the perverse sheaf $\Z^\bullet_X[n]$; $\alpha$ is the composition of the canonical maps
$$
{}^{\mu}\hskip -0.02in H^0(l_!\Z^\bullet_{X\backslash\Sigma}[n])\cong {}^{\mu}\hskip -0.02in H^0(l_!l^!\Z^\bullet_X[n])\xrightarrow{\beta}\Z^\bullet_X[n]\xrightarrow{\gamma}{}^{\mu}\hskip -0.02in H^0(l_*l^*\Z^\bullet_X[n])\cong {}^{\mu}\hskip -0.02in H^0(l_*\Z^\bullet_{X\backslash\Sigma}[n]).
$$

We claim that $\beta$ is a surjection and, hence, $\operatorname{im}\gamma\cong \Idot_X$. To see this, take the canonical distinguished triangle
$$
\rightarrow l_!l^!\Z^\bullet_X[n]\rightarrow \Z^\bullet_X[n]\rightarrow m_*m^*\Z^\bullet_X[n]\xrightarrow{[1]}
$$
and consider a portion of the long exact sequence in $\operatorname{Perv}(X)$ obtained by applying perverse cohomology:
$$
\rightarrow {}^{\mu}\hskip -0.02in H^0(l_!l^!\Z^\bullet_X[n])\xrightarrow{\beta}\Z^\bullet_X[n]\rightarrow {}^{\mu}\hskip -0.02in H^0(m_*m^*\Z^\bullet_X[n])\rightarrow.
$$
We want to show that ${}^{\mu}\hskip -0.02in H^0(m_*m^*\Z^\bullet_X[n])=0$. 

We have
$$
{}^{\mu}\hskip -0.02in H^0(m_*m^*\Z^\bullet_X[n])\cong {}^{\mu}\hskip -0.02in H^0(m_*\Z^\bullet_\Sigma[n]).
$$

\smallskip

Then it is trivial that the complex $\Z^\bullet_\Sigma[s]$ satisfies the support condition, and so ${}^{\mu}\hskip -0.02in H^k(\Z^\bullet_\Sigma[s])=0$ for $k\geq 1$. But 
$$
{}^{\mu}\hskip -0.02in H^0(m_*\Z^\bullet_\Sigma[n])\cong m_*{}^{\mu}\hskip -0.02in H^{n-s}(\Z^\bullet_\Sigma[s]),
$$
which equals $0$ since $n-s\geq 1$. Therefore, $\beta$ is a surjection and $\operatorname{im}\gamma\cong \Idot_X$.

Now take the canonical distinguished triangle
$$
\rightarrow m_!m^!\Z^\bullet_X[n]\rightarrow \Z^\bullet_X[n]\rightarrow l_*l^*\Z^\bullet_X[n]\xrightarrow{[1]}
$$
and consider a portion of the long exact sequence in $\operatorname{Perv}(X)$ obtained by applying perverse cohomology:
$$
\rightarrow {}^{\mu}\hskip -0.02in H^{-1}(l_*l^*\Z^\bullet_X[n])\rightarrow {}^{\mu}\hskip -0.02in H^0(m_!m^!\Z^\bullet_X[n])\rightarrow\Z^\bullet_X[n]\xrightarrow{\gamma} {}^{\mu}\hskip -0.02in H^0(l_*l^*\Z^\bullet_X[n])\rightarrow.
$$

We claim that ${}^{\mu}\hskip -0.02in H^{-1}(l_*l^*\Z^\bullet_X[n])=0$. This is easy; as $\Z^\bullet_X[n]$ is perverse, $l^*\Z^\bullet_X[n]$ is perverse (since it is the restriction to an open subset) and, in particular, satisfies the cosupport condition. By 10.3.3.iv of \cite{kashsch}, $l_*l^*\Z^\bullet_X[n]$ also satisfies the cosupport condition. Thus, ${}^{\mu}\hskip -0.02in H^{-1}(l_*l^*\Z^\bullet_X[n])=0$.

Therefore, ${}^{\mu}\hskip -0.02in H^0(m_!m^!\Z^\bullet_X[n])$ is the kernel of the map $\gamma$, whose image is $\Idot_X$, i.e., this kernel is how we defined $\Ndot_X$, and we are finished.
\end{proof}

Now we can prove the first theorem that we stated in the introduction, \thmref{intro:thm1}.

\begin{thm}\label{thm:1theorem} In $\operatorname{Perv}(X)$, there is an isomorphism
$$
\Ndot_X\cong \operatorname{ker}\big\{\operatorname{id}-\widetilde T_f\big\}.
$$
Dually, there is an isomorphism between $\operatorname{coker}\big\{\operatorname{id}-\widetilde T_f\big\}$ and the cokernel of the canonical injection $\Idot_X\hookrightarrow j^![1]\Z^\bullet_\U[n+1]$.

\end{thm}
\begin{proof}
Consider the two nearby-vanishing short exact sequences from the beginning of the section:
$$
0\rightarrow j^*[-1]\Z^\bullet_\U[n+1]\rightarrow\psi_f[-1]\Z^\bullet_\U[n+1]\xrightarrow{\operatorname{can}}\phi_f[-1]\Z^\bullet_\U[n+1]\rightarrow 0
$$
and
$$
0\rightarrow \phi_f[-1]\Z^\bullet_\U[n+1]\xrightarrow{\operatorname{var}}\psi_f[-1]\Z^\bullet_\U[n+1]\rightarrow j^![1]\Z^\bullet_\U[n+1]\rightarrow 0
$$

\bigskip

\noindent and apply $m_!m^!$ to obtain two distinguished triangles:
$$
\rightarrow m_!m^!j^*[-1]\Z^\bullet_\U[n+1]\rightarrow m_!m^!\psi_f[-1]\Z^\bullet_\U[n+1]\xrightarrow{\operatorname{m_!m^!\operatorname{can}}}m_!m^!\phi_f[-1]\Z^\bullet_\U[n+1]\xrightarrow{[1]}
$$
and
$$
\rightarrow m_!m^!\phi_f[-1]\Z^\bullet_\U[n+1]\xrightarrow{\operatorname{m_!m^!\operatorname{var}}}m_!m^!\psi_f[-1]\Z^\bullet_\U[n+1]\rightarrow m_!m^!j^![1]\Z^\bullet_\U[n+1]\xrightarrow{[1]}.
$$

\smallskip

As $j^*[-1]\Z^\bullet_\U[n+1]\cong\Z^\bullet_X[n]$ and the support of $\phi_f[-1]\Z^\bullet_\U[n+1]$ is $\Sigma$, these distinguished triangles become
\begin{equation}
\rightarrow m_!m^!\Z^\bullet_X[n]\rightarrow m_!m^!\psi_f[-1]\Z^\bullet_\U[n+1]\xrightarrow{m_!m^!\operatorname{can}}\phi_f[-1]\Z^\bullet_\U[n+1]\xrightarrow{[1]}\tag{$\dagger$}
\end{equation}

and

\begin{equation}
\rightarrow \phi_f[-1]\Z^\bullet_\U[n+1]\xrightarrow{\operatorname{m_!m^!\operatorname{var}}}m_!m^!\psi_f[-1]\Z^\bullet_\U[n+1]\rightarrow m_!\hat m^![1]\Z^\bullet_\U[n+1]\xrightarrow{[1]}.\tag{$\ddagger$}
\end{equation}

By applying perverse cohomology to ($\dagger$), using that $\phi_f[-1]\Z^\bullet_\U[n+1]$ is perverse, and using the lemma, we immediately conclude that
$$
\Ndot_X\cong {}^{\mu}\hskip -0.02in H^0(m_!m^!\Z^\bullet_X[n])\cong \operatorname{ker}\big\{ {}^{\mu}\hskip -0.02in H^0m_!m^!\operatorname{can}\big\}.
$$

Now, it is an easy exercise to verify that $\hat m^!\Z^\bullet_\U[2n+2-s]$ satisfies the cosupport condition. This implies that ${}^{\mu}\hskip -0.02in H^k(\hat m^!\Z^\bullet_\U[2n+2-s])=0$ for all $k\leq -1$, i.e., 
$$
{}^{\mu}\hskip -0.02in H^k(m_!\hat m^![1]\Z^\bullet_\U[n+1])\cong m_!{}^{\mu}\hskip -0.02in H^k(\hat m^![1]\Z^\bullet_\U[n+1])=0, \textnormal{ for }k\leq -1+(n-s).
$$
As $n-s\geq 1$, ${}^{\mu}\hskip -0.02in H^k(m_!\hat m^![1]\Z^\bullet_\U[n+1])$ is zero for $k\leq 0$. Therefore, applying perverse cohomology to ($\ddagger$), yields that
$$
\phi_f[-1]\Z^\bullet_\U[n+1]\xrightarrow{\operatorname{ {}^{\mu}\hskip -0.02in H^0m_!m^!\operatorname{var}}}{}^{\mu}\hskip -0.02in H^0(m_!m^!\psi_f[-1]\Z^\bullet_\U[n+1])
$$
is an isomorphism.

Finally, we have
$$
\operatorname{ker}\big\{\operatorname{id}-\widetilde T_f\big\}\cong\operatorname{ker}\big\{{}^{\mu}\hskip -0.02in H^0m_!m^!(\operatorname{id}-\widetilde T_f)\big\}\cong \operatorname{ker}\big\{{}^{\mu}\hskip -0.02in H^0m_!m^!\operatorname{can}\circ \,{}^{\mu}\hskip -0.02in H^0m_!m^!\operatorname{var}\big\}\cong
$$
$$
\cong \operatorname{ker}\big\{ {}^{\mu}\hskip -0.02in H^0m_!m^!\operatorname{can}\big\}\cong \Ndot_X.
$$

\smallskip

The dual statement follows by dualizing all of the above arguments.
\end{proof}

\smallskip

\begin{rem}\label{rem:field} While we have proved \thmref{thm:1theorem} with integral coefficients, the proof of the analogous statement with coefficients in an arbitrary field $\mathfrak K$ is precisely the same. That is, if $\Ndot_X(\mathfrak K)$ is defined as the kernel of the canonical surjection from the constant sheaf $\mathfrak K^\bullet_X[n]$ to the intersection cohomology with constant $\mathfrak K$ coefficients $\Idot_X(\mathfrak K)$, then $\Ndot_X(\mathfrak K)$ is isomorphic to the kernel of
$$
\operatorname{id}-\widetilde T^{\mathfrak K}_f: \phi_f[-1]\mathfrak K^\bullet_\U[n+1]\rightarrow\phi_f[-1]\mathfrak K^\bullet_\U[n+1].
$$

With field coefficients, we proved a simple corollary of \thmref{thm:1theorem} in Theorem 3.2 of \cite{icmono}, where we showed that $\operatorname{ker}\big\{\operatorname{id}-\widetilde T^{\mathfrak K}_f\big\}=0$ if and only if $\mathfrak K^\bullet_X[n]$ is isomorphic to $\Idot_X(\mathfrak K)$. 
\end{rem}

\smallskip

As an immediate, somewhat surprising, corollary of the fact that intersection cohomology is a topological invariant, we have:
\begin{cor} The perverse sheaf $\operatorname{ker}\big\{\operatorname{id}-\widetilde T_f\big\}$ is an invariant of the (non-embedded) topological-type of the hypersurface $V(f)$.

Precisely, suppose that $f:\U\rightarrow\C$ and $g:\U\rightarrow\C$ are reduced complex analytic functions, let $X:=V(f)$ and $Y:=V(g)$, and suppose that $H:X\rightarrow Y$ is a homeomorphism. Then,
$$
H\big(\operatorname{supp}\big(\operatorname{ker}\big\{\operatorname{id}-\widetilde T_f\big\}\big)\big)=\operatorname{supp}\big(\operatorname{ker}\big\{\operatorname{id}-\widetilde T_g\big\}\big)
$$
and
$$
H_*\big(\operatorname{ker}\big\{\operatorname{id}-\widetilde T_f\big\}\big)\cong \operatorname{ker}\big\{\operatorname{id}-\widetilde T_g\big\}.
$$
\end{cor}

\bigskip

In the next section, we look at the other eigenspaces of the monodromy. As we shall see, that case is much simpler than the $\xi=1$ case.

\section{The Main Theorem for $\xi\neq0,1$}\label{sec:mainnot1}
Throughout this section, we use $\C$ as our base ring.

We need to look at the germ of the situation at a point $p\in X$. So, fix a $p\in X$, and choose real $\epsilon$ and $\delta$, $0<\delta\ll\epsilon\ll 1$, such that $B^\circ_\epsilon(p)\subseteq\U$ and
$$
B^\circ_\epsilon(p)\cap f^{-1}(\D^\circ_\delta\backslash\{\0\})\xrightarrow{\hat f}\D^\circ_\delta\backslash\{\0\}
$$
is a locally trivial fibration, whose fiber is the Milnor fiber of $f$ at $p$, where $\hat f$ denotes the restriction of $f$. That this can be done is a standard result on the Milnor fibration inside a ball. Replace the open set $\U$ with the open set $B^\circ_\epsilon(p)\cap f^{-1}(\D^\circ_\delta)$. On the open dense subset $B^\circ_\epsilon(p)\cap f^{-1}(\D^\circ_\delta\backslash\{\0\})$, we take the local system $\mathcal L_\xi$, which is in degree $-(n+1)$, has stalk cohomology $\C$ in degree $-(n+1)$ and is given by the representation 
$$\pi_1\big(B^\circ_\epsilon(p)\cap f^{-1}(\D^\circ_\delta\backslash\{\0\})\big)\xrightarrow{\hat f_*}\pi_1(\D^\circ_\delta\backslash\{\0\})\cong \Z\xrightarrow{h}\operatorname{Aut}(\C),$$
where $\D^\circ_\delta\backslash\{\0\}$ is oriented counterclockwise and  $h$ is the homomorphism which takes the generator $<1>$ to multiplication by $\xi$. Thus, $\mathcal L_\xi$ is the rank $1$ local system, in degree $-(n+1)$, which multiplies by $\xi$ as one goes once around a counterclockwise meridian around the hypersurface $X$.

As in the introduction, we let $\Idot_\U(\xi)$ denote intersection cohomology using the local system $\mathcal L_\xi$; this is the intermediate extension of $\mathcal L_\xi$ to all of $\U$. It is a simple object in the category of perverse sheaves; see \cite{BBD}. Then, it is well-known that $j^*[-1]\Idot_\U(\xi)$ and $j^![1]\Idot_\U(\xi)$ are perverse, but, for lack of a convenient reference, we give a quick argument for this.

In the derived category, consider the two canonical distinguished triangles:

$$
j_*j^*[-1]\Idot_\U(\xi)\rightarrow i_!i^!\Idot_\U(\xi)\rightarrow \Idot_\U(\xi)\xrightarrow{[1]}
$$
and
$$
\Idot_\U(\xi)\rightarrow  i_*i^*\Idot_\U(\xi)\rightarrow j_!j^![1]\Idot_\U(\xi)\xrightarrow{[1]}.
$$

Now, $i^!\Idot_\U(\xi)\cong i^*\Idot_\U(\xi)\cong \mathcal L_\xi$, and $i_!\mathcal L_\xi$ and $i_*\mathcal L_\xi$ are well-known to be perverse (by Proposition 10.3.3 and 10.3.17 of \cite{kashsch}, combined with Theorem 5 of Section 5.1 of \cite{grauertremmertstein}). It follows immediately from this, and the simplicity of $\Idot_\U(\xi)$, that $j_*j^*[-1]\Idot_\U(\xi)$ and $j_!j^![1]\Idot_\U(\xi)$ are perverse. As $j_*\cong j_!$ is simply extension by zero, we conclude that $j^*[-1]\Idot_\U(\xi)$ and $j^![1]\Idot_\U(\xi)$ are perverse.

\medskip

Now we can prove the main theorem for $\xi\neq 0$ or $1$:

\begin{thm}\label{thm:main2} Using $\C$ as our base ring, and supposing the $\xi\neq 0$ or $1$, there is an isomorphism of perverse sheaves 
$$j^*[-1]\Idot_\U(\xi)\cong\ker\{\xi^{-1}\operatorname{id}- T_f\}\cong \ker\{\xi^{-1}\operatorname{id}-\widetilde T_f\}.$$

Dually, 
$$j^![1]\Idot_\U(\xi)\cong\coker\{\xi^{-1}\operatorname{id}- T_f\}\cong \coker\{\xi^{-1}\operatorname{id}-\widetilde T_f\}.$$
\end{thm}
\begin{proof} We prove the first statement, and leave the dual argument to the reader.

As the nearby cycles along $f$ are determined by perturbing $f$ in a radial direction in $\D^\circ_\delta\backslash\{\0\}$, it follows that there is an isomorphism 
$$Q:\psi_f[-1]\C^\bullet_\U[n+1]\rightarrow \psi_f[-1]\Idot_\U(\xi);$$
 moreover, by construction, this isomorphism identifies the monodromy $T_f(\xi)$ on $\psi_f[-1]\Idot_\U(\xi)$ with $\xi\cdot T_f$ on $\psi_f[-1]\C^\bullet_\U[n+1]$, i.e., 
 $$
 \xi\cdot T_f = Q^{-1}\circ T_f(\xi)\circ Q.
 $$
 Therefore, 
 $$\ker\{\operatorname{id}-T_f(\xi)\}\cong \ker\{\operatorname{id}-\xi\cdot T_f\}\cong  \ker\{\xi^{-1}\operatorname{id}-T_f\}.
 $$
 
 \smallskip
 
 Consider again the two nearby-vanishing short exact sequences:
$$
0\rightarrow j^*[-1]\Idot_\U(\xi)\rightarrow\psi_f[-1]\Idot_\U(\xi)\xrightarrow{\operatorname{can}}\phi_f[-1]\Idot_\U(\xi)\rightarrow 0
$$
and
$$
0\rightarrow \phi_f[-1]\Idot_\U(\xi)\xrightarrow{\operatorname{var}}\psi_f[-1]\Idot_\U(\xi)\rightarrow j^![1]\Idot_\U(\xi)\rightarrow 0,
$$

\medskip

\noindent where $\operatorname{var}\circ\operatorname{can}\cong \operatorname{id}-T_f(\xi)$. We immediately conclude that 
$$j^*[-1]\Idot_\U(\xi)\cong \ker\{\operatorname{id}-T_f(\xi)\}\cong  \ker\{\xi^{-1}\operatorname{id}-T_f\}.
$$

Finally, consider the triple of endomorphisms 
$$(\xi^{-1}\operatorname{id}-\operatorname{id},\, \xi^{-1}\operatorname{id}-T_f,\, \xi^{-1}\operatorname{id}-\widetilde T_f)$$
which acts on the short exact sequence
$$
0\rightarrow j^*[-1]\C^\bullet_\U[n+1]\rightarrow\psi_f[-1]\C^\bullet_\U[n+1]\xrightarrow{\operatorname{can}}\phi_f[-1]\C^\bullet_\U[n+1]\rightarrow 0
$$
and commutes with its maps. Then we have the associated long exact sequence

\medskip

\noindent$0\rightarrow\ker\{(\xi^{-1}-1)\operatorname{id}\}\rightarrow \ker\{\xi^{-1}\operatorname{id}- T_f\}\rightarrow\hfill$

\smallskip

\noindent$\hfill \ker\{\xi^{-1}\operatorname{id}- \widetilde T_f\}\rightarrow\coker\{(\xi^{-1}-1)\operatorname{id}\}\rightarrow\cdots .$

\medskip

As  $\ker\{(\xi^{-1}-1))\operatorname{id}\}$ and $\coker\{(\xi^{-1}-1)\operatorname{id})\}$   are zero,  we conclude the final isomorphism of the theorem, that   
$$
\ker\{\xi^{-1}\operatorname{id}- T_f\}\cong \ker\{\xi^{-1}\operatorname{id}-\widetilde T_f\}.
$$               
\end{proof}

\smallskip

\begin{cor}\label{cor:main2} If $\xi\neq 0, 1$, then there a short exact sequence in $\operatorname{Perv}(\U)$
$$
0\rightarrow  j_*\ker\{\xi^{-1}\operatorname{id}-\widetilde T_f\}\rightarrow i_!\mathcal L_\xi\xrightarrow{\omega_\xi} \Idot_\U(\xi)\rightarrow 0,
$$
where $\omega_\xi$ is the canonical surjection.

In particular, for $x\in X$, for all $k$, there are isomorphisms on stalk cohomology
$$
H^k\big(\ker\{\xi^{-1}\operatorname{id}-\widetilde T_f\}\big)_x\cong H^{k-1}\big(\Idot_\U(\xi)\big)_x.
$$
\end{cor}

\section{Applications}\label{sec:appl}

What we have proved in \thmref{thm:1theorem} and \thmref{thm:main2} is that there is a short exact sequence in $\operatorname{Perv}(X)$
\begin{equation}
0\rightarrow \operatorname{ker}\big\{\operatorname{id}-\widetilde T_f\big\}\rightarrow\Z^\bullet_X[n]\xrightarrow{\tau_{{}_X}}\mathbf I^\bullet_X\rightarrow0,\tag{$*$}
\end{equation}
where $\tau_X$ is the canonical surjection, and, if $\xi\neq 0, 1$, a short exact sequence in $\operatorname{Perv}(\U)$
\begin{equation}
0\rightarrow  j_*\ker\{\xi^{-1}\operatorname{id}-\widetilde T_f\}\rightarrow i_!\mathcal L_\xi\xrightarrow{\omega_\xi} \Idot_\U(\xi)\rightarrow 0,\tag{$**$}
\end{equation}
where $\omega_\xi$ is the canonical surjection.

The reader should naturally ask: what does this tell us about the topology of hypersurface singularities {\bf without} the language of the derived category and $\operatorname{Perv}(X)$?

\smallskip 

It is {\bf not} true, in general, that cohomology of the stalk of the kernel is isomorphic to the kernel of the cohomology on the stalks, i.e., there may exist $x\in \Sigma$ and a degree $k$ such that
$$H^k\big(\operatorname{ker}\big\{\operatorname{id}-\widetilde T_f\big\}\big)_x\not\cong \operatorname{ker}\big\{\operatorname{id}-(\widetilde T_f)^k_x\big\},
$$
where $(\widetilde T_f)^k_x$ is the induced map on $H^k(\phi_f[-1]\Z^\bullet_\U[n+1])_x$. 
This makes the relationship with the topological theory complicated, since it is $(\widetilde T_f)_x$ which is the classical Milnor monodromy on the reduced cohomology of the Milnor fiber of $f$ at $x$. 

\smallskip

However, we have the following known, easy lemma, which we prove for lack of a convenient reference.

\begin{lem}\label{lem:onedegree} Suppose that $G:\Adot\rightarrow\Bdot$ is a morphism of perverse sheaves on an analytic space $Y$. Let $y\in Y$, and let $d:=\dim_yY$. Then, there is an isomorphism
$$
H^{-d}\big(\operatorname{ker} G\big)_y\cong \operatorname{ker}\left\{G^{-d}_y: H^{-d}(\Adot)_y\rightarrow H^{-d}(\Bdot)_y\right\}.
$$
\end{lem}
\begin{proof} Consider the canonical short exact sequences in $\operatorname{Perv}(Y)$:
$$
0\rightarrow\operatorname{ker}G\rightarrow\Adot\xrightarrow{\ \alpha\ }\operatorname{im}G\rightarrow 0
$$
and
$$
0\rightarrow\operatorname{im}G\xrightarrow{\ \beta\ }\Bdot\rightarrow\operatorname{coker}G\rightarrow 0,$$
where $\beta\circ\alpha=G$.

Since the stalk cohomology at $y$ of all of these perverse sheaves is zero below degree $-d$, the non-zero terms of the associated long exact sequences on stalk cohomology begin as follows:
$$
0\rightarrow H^{-d}(\operatorname{ker}G)_y\rightarrow H^{-d}(\Adot)_y\xrightarrow{\alpha^{-d}_y}H^{-d}(\operatorname{im}G)_y\rightarrow\dots
$$
and
$$
0\rightarrow H^{-d}(\operatorname{im}G)_y\xrightarrow{\beta^{-d}_y} H^{-d}(\Bdot)_y\rightarrow\dots\ .$$

\medskip

The conclusion is immediate, since $\beta_y^{-d}\circ\alpha_y^{-d} = G^{-d}_y$.
\end{proof}

\medskip

\begin{rem} Of course we intend to apply \lemref{lem:onedegree} to the case where $G=\xi^{-1}\operatorname{id}-\widetilde T_f$ (including the case where $\xi$ may be $1$). While $\phi_f[-1]\C_\U[n+1]$ (or with $\Z$ coefficients) is a perverse sheaf on all of $X$, its support is $\Sigma$, and we may (and will) apply \lemref{lem:onedegree} after restricting, without further comment, to $\Sigma$.
\end{rem}

\medskip

\begin{rem} The result of the lemma should be contrasted with our result in Theorem 3.1 of \cite{pervkernel}, where we looked at an endomorphism of perverse sheaves, $T:\Pdot\rightarrow\Pdot$, with a field as the base ring. However, there, we looked at the kernel in degree $-e$, where $e:=\dim\operatorname{supp}(\operatorname{ker})$, which could be strictly less than what we consider in the lemma, which is essentially $\dim\operatorname{supp}\Pdot$.
\end{rem}

\medskip

Below, we will, as in the introduction, use topological indexing for the reduced intersection cohomology $\widetilde{IH}$.

\smallskip

From  \thmref{thm:1theorem} and the lemma, we conclude:

\begin{prop}\label{prop:onedegree} Let $x\in\Sigma$ and let $s:=\dim_x\Sigma$. Let
$$
 H^{n-s}(F_{f, x};\,\Z)\xrightarrow{(\widetilde T_f)^{-s}_x} H^{n-s}(F_{f, x};\,\Z)
$$
be the standard Milnor monodromy on the degree $(n-s)$ cohomology of the Milnor fiber of $f$ at $x$.

Then, there are isomorphisms
$$
\operatorname{ker}\big\{\operatorname{id}- (\widetilde T_f)^{-s}_x\big\}\cong \Z^\omega\cong \widetilde{IH}^{n-s-1}(B_\epsilon^\circ(x)\cap X; \, \Z),
$$
\medskip
\noindent where $0<\epsilon\ll 1$ and
$$
\omega:=\begin{cases}
\phantom{-1+\,\, }\operatorname{rank} H^{n+s}(K_{X, x};\,\Z), &\textnormal{ if } s\neq n-1;\\
-1+\operatorname{rank} H^{n+s}(K_{X, x};\,\Z), &\textnormal{ if } s= n-1.
\end{cases}
$$
\end{prop}
\begin{proof} Consider a portion of the long exact sequence on stalk cohomology which comes from the short exact sequence ($\ast$), and apply \lemref{lem:onedegree}; we obtain:
$$
0\rightarrow H^{n-s-1}(\Z^\bullet_X)_x\rightarrow H^{-s-1}(\Idot)_x\rightarrow \operatorname{ker}\big\{\operatorname{id}- (\widetilde T_f)^{-s}_x\big\}\rightarrow 0.
$$
Note that $H^{n-s-1}(\Z^\bullet_X)_x=0$ unless $s=n-1$.

Therefore, 
$$
\operatorname{ker}\big\{\operatorname{id}- (\widetilde T_f)^{-s}_x\big\}\cong \widetilde{IH}^{n-s-1}(B_\epsilon^\circ(x)\cap X; \, \Z).
$$

Now, the fact that $H^{n-s}(F_{f, x};\,\Z)$ is free Abelian is classical; it follows from the connectivity result of Kato and Matsumoto \cite{katomatsu}, the Hurewicz Theorem,  and the Universal Coefficient Theorem. Thus, $\operatorname{ker}\big\{\operatorname{id}- (\widetilde T_f)^{-s}_x\big\}$ is free Abelian.

Since our remaining claim is just about the value of $\omega$, it suffices for us to work over a base ring that is the field $\C$, rather than $\Z$. With field coefficients, $\Idot_X$ is self-(Verdier) dual, i.e., $\mathcal D\Idot_X\cong\Idot_X$.

Let $j_x$ denote the inclusion of $\{x\}$ into $X$. Then,
$$
H^{-s-1}(\Idot)_x=H^{-s-1}(j_x^*\Idot_X)\cong H^{-s-1}(\mathcal D j_x^!\mathcal D \Idot_x)\cong H^{s+1}(j_x^!\Idot)\cong H^{s+1}(j^!_x\C^\bullet_X[n]),
$$
where the last isomorphism follows by applying $j^!_x$ to the short exact sequence defining $\Ndot_X$ and taking the associated long exact sequence.

Now,
$$
H^{s+1}(j^!_x\C^\bullet_X[n])\cong H^{n+s+1}(B_\epsilon^\circ(x), B_\epsilon^\circ(x)\backslash\{x\};\, \C)\cong
$$
$$
H^{n+s}(B_\epsilon^\circ(x)\backslash\{x\};\, \C)\cong H^{n+s}(K_{X, x};\, \C).
$$
This completes the proof.
\end{proof} 

\medskip

\begin{rem} We should point out that one does not need \thmref{thm:1theorem} to prove the ordinary cohomology statement in \propref{prop:onedegree}. Milnor's work in Section 8 of \cite{milnorsing} tells us how the homology/cohomology of the complement of the real link $K_{X, x}$ inside $S_\epsilon^{2n+1}$ relates to the kernel of $\operatorname{id}-(\widetilde T_f)^{-s}_x$. Then, using Alexander Duality, one can recover  the first isomorphism in \propref{prop:onedegree}.

Also, related to the case where $s=n-1$, it is  well-known that the rank of $H^{2n-1}(K_{X, x};\,\Z)$ is equal to the rank of $IH^0(B_\epsilon^\circ(x)\cap X; \, \Z)$, which is equal to the number of irreducible components of $X$ at $x$.
\end{rem}

\bigskip

The proposition with $\xi\neq0, 1$ which corresponds to \propref{prop:onedegree}  is essentially impossible to state without at least using intersection cohomology with twisted coefficients. Still, we find it interesting that one immediately concludes from \corref{cor:main2} and \lemref{lem:onedegree} that:

\begin{prop}  Let $x\in\Sigma$ and let $s:=\dim_x\Sigma$. Let
$$
 H^{n-s}(F_{f, x};\,\C)\xrightarrow{(\widetilde T_f)^{-s}_x} H^{n-s}(F_{f, x};\,\C)
$$
be the standard Milnor monodromy on the degree $(n-s)$ cohomology of the Milnor fiber of $f$ at $x$, and let $\xi\neq 0, 1$.

Then, there is an isomorphism
$$
\operatorname{ker}\big\{\xi^{-1}\operatorname{id}- (\widetilde T_f)^{-s}_x\big\}\cong H^{-s-1}\big(\Idot_\U(\xi)\big)_x.
$$

\end{prop} 

\bigskip

Another application of the short exact sequence ($*$) is that, for each $x\in X$, we can apply the vanishing cycle functor, $\phi_L[-1]$, which is exact on $\operatorname{Perv}(X)$, where $L$ is the restriction to $X$ of a generic affine linear form such that $L(x)=0$. We then obtain a short exact sequence of $\Z$-modules:
\begin{equation}
0\rightarrow \operatorname{ker}\big\{\operatorname{id}-(\phi_L[-1]\widetilde T_f)^0_x\big\}\rightarrow H^0\big(\phi_L[-1]\Z^\bullet_X[n]\big)_x\rightarrow H^0\big(\phi_L[-1]\mathbf I^\bullet_X\big)_x\rightarrow0,\tag{$\Diamond$}
\end{equation}
where $(\phi_L[-1]\widetilde T_f)^0_x$ is the automorphism induced by the $f$-monodromy on $$H^0\big(\phi_L[-1]\phi_f[-1]\Z^\bullet_\U[n+1]\big)_x\cong\Z^{\lambda^0_{f, L}(x)},$$
where $\lambda^0_{f, L}(x)$ is the $0$-th L\^e number of $f$ with respect to $L$ at $x$ (see \cite{lecycles}).

Now, $\operatorname{rank}H^0\big(\phi_L[-1]\mathbf I^\bullet_X\big)_x$ is the coefficient of $\{x\}$ in the characteristic cycle of intersection cohomology; this is a great importance in some settings (see, for instance, \cite{bradenpams} for a discussion). (There are different shifting/sign conventions on the characteristic cycle; we are using a convention that, for a perverse sheaf, gives us that all of the coefficients are non-negative.) However, without the language of the derived category and using topological indexing, $H^0\big(\phi_L[-1]\mathbf I^\bullet_X\big)_x$ is isomorphic to
$$
\operatorname{coker}\left\{\widetilde{IH}^{n-1}(B^\circ_\epsilon(x)\cap X;\, \Z)\xrightarrow{r_{{}_{X, x}}}\widetilde{IH}^{n-1}(B^\circ_\epsilon(x)\cap X\cap L^{-1}(\gamma);\, \Z)\right\},
$$
where $r_{{}_{X, x}}$ is induced by the restriction and, as before, $L$ is a generic affine linear form such that $L(x)=0$ and $0<|\gamma|\ll \epsilon\ll 1$.

Furthermore, as the result of L\^e in  \cite{levan} tells us that the complex link $\cL_{X,x}$ of $X$ at $x$ has the homotopy-type of a bouquet of $(n-1)$-spheres, the number of spheres in this homotopy-type is precisely $\operatorname{rank} H^0\big(\phi_L[-1]\Z^\bullet_X[n]\big)_x$, which is known to equal  the intersection number $\big(\Gamma_{f, L}^1\cdot V(L)\big)_x$, where $\Gamma_{f, L}^1$ is the relative polar curve of $f$ with respect to $L$; see Corollary 2.6 of \cite{gencalcvan} (though this was known earlier by Hamm, L\^e, Siersma, and Teissier). 

Now, without the language of the derived category,  $\operatorname{ker}\big\{\operatorname{id}-\phi_L[-1](\widetilde T^0_f)_x\big\}$ would be the kernel of an endomorphism on the cohomology of a pair of pairs of spaces; this, again, is a complicated object. However, the dual argument to that appearing near the end of the proof of \thmref{thm:1theorem} tells us that
$$
{}^{\mu}\hskip -0.02in H^0(m_*m^\ast\psi_f[-1]\Z^\bullet_\U[n+1])\xrightarrow{\operatorname{ {}^{\mu}\hskip -0.02in H^0m_*m^\ast\operatorname{can}}}\phi_f[-1]\Z^\bullet_\U[n+1]
$$
is an isomorphism.

As $\phi_L[-1]$ is $t$-exact (and as $L$ is generic), we have 
$$
\phi_L[-1]\phi_f[-1]\Z^\bullet_\U[n+1]\cong \phi_L[-1]{}^{\mu}\hskip -0.02in H^0(m_*m^*\psi_f[-1]\Z^\bullet_\U[n+1])\cong
$$
$$
{}^{\mu}\hskip -0.02in H^0\big(\phi_L[-1](m_*m^*\psi_f[-1]\Z^\bullet_\U[n+1])\big),
$$
which are perverse sheaves with $x$ as an isolated point in their support. 

Thus, 
$$
H^0\big(\phi_L[-1]\phi_f[-1]\Z^\bullet_\U[n+1]\big)_x\cong H^0\big(\phi_L[-1](m_*m^*\psi_f[-1]\Z^\bullet_\U[n+1])\big)_x,
$$
by an isomorphism which takes $\operatorname{id}-(\phi_L[-1]\widetilde T_f)^0_x$ to $\operatorname{id}-(\phi_L[-1]m_*m^*T_f)^0_x$.

\medskip

Finally, we conclude that the short exact sequence ($\Diamond$) tells us:

\begin{cor}\label{cor:vanvan} There is an equality
$$
\operatorname{rank}\operatorname{coker}\{r_{{}_{X, x}}\}=\big(\Gamma_{f, L}^1\cdot V(L)\big)_x-\operatorname{rank}\operatorname{ker}\big\{\operatorname{id}-\widehat T_{f, L}\big\},
$$
where  
$$
r_{{}_{X, x}}: \widetilde{IH}^{n-1}(B^\circ_\epsilon(x)\cap X;\, \Z)\rightarrow\widetilde{IH}^{n-1}(B^\circ_\epsilon(x)\cap X\cap L^{-1}(\gamma);\, \Z),
$$
is induced by restriction, and where $\widehat T_{f, L}$ is the automorphism induced by the Milnor monodromy of $f$ on the relative cohomology
$$
H^n\big(B^\circ_\epsilon(x)\cap f^{-1}(\eta), \,D^\circ_\delta\cap f^{-1}(\eta); \,\Z\big),
$$
where $0<|\eta|\ll \delta\ll |\gamma|\ll\epsilon\ll 1$ and $D^\circ_\delta$ is the union of all the open balls of radius $\delta$ centered at each of the points in $B^\circ_\epsilon(x)\cap \Sigma\cap L^{-1}(\gamma)$.
\end{cor}

\smallskip

In the case where $s=0$, \corref{cor:vanvan} reduces to \propref{prop:easycool}. The case where $s=1$ simplifies enough that it is worth stating separately.

\begin{cor} Suppose that $\dim_x\Sigma=1$ and that $\dim_x\Sigma(f_{|_{V(L)}})=0$. Then, there is an equality
$$
\operatorname{rank}\operatorname{coker}\{r_{{}_{X, x}}\}=\big(\Gamma_{f, L}^1\cdot V(L)\big)_x-\operatorname{rank}\operatorname{ker}\big\{\operatorname{id}-\widehat T_{f, L}\big\},
$$
where  
$$
r_{{}_{X, x}}: {IH}^{n-1}(B^\circ_\epsilon(x)\cap X;\, \Z)\rightarrow{IH}^{n-1}(B^\circ_\epsilon(x)\cap X\cap L^{-1}(\gamma);\, \Z),
$$
is induced by restriction, and where $\widehat T_{f, L}$ is the automorphism induced by the Milnor monodromy of $f$ on the relative cohomology
$$
H^n\big(F_{f, x}, \,\dot{\cup} F_{f, x_i}; \,\Z\big),
$$
where the union is over $x_i\in B^\circ_\epsilon(x)\cap \Sigma\cap L^{-1}(\gamma)$, where $0< |\gamma|\ll\epsilon\ll 1$, and $F_{f, x}$ and $F_{f, x_i}$ denote the Milnor fibers of $f$ at the respective points.

\end{cor}

\bigskip

Now, we want to address the $\xi\neq 0, 1$ case. For each $x\in X$, we can also apply to the short exact sequence $(**)$ the vanishing cycle functor, $\phi_L[-1]$, where $L$ is a generic affine linear form on $\U$ such that $L(x)=0$.  We then obtain a short exact sequence of $\C$-vector spaces:
$$
0\rightarrow \operatorname{ker}\big\{\xi^{-1}\operatorname{id}-(\phi_L[-1]\widetilde T_f)^0_x\big\}\rightarrow H^0\big(\phi_L[-1]i_!\mathcal L_\xi\big)_x\rightarrow H^0\big(\phi_L[-1] \Idot_\U(\xi)\big)_x\rightarrow0,
$$
where $(\phi_L[-1]\widetilde T_f)^0_x$ is the automorphism induced by the $f$-monodromy on $$H^0\big(\phi_L[-1]\phi_f[-1]\C^\bullet_\U[n+1]\big)_x \cong\C^{\lambda^0_{f, L}(x)}.$$

Now, it follows from Theorem 4.2.B of \cite{hypercohom} that 
$$\dim H^0\big(\phi_L[-1]i_!\mathcal L_\xi\big)_x= \big(\Gamma_{f, L}^1\cdot V(L)\big)_x.
$$

Thus, we obtain:
\begin{cor}\label{cor:xi} There is an equality 
$$
\big(\Gamma_{f, L}^1\cdot V(L)\big)_x-\dim \operatorname{ker}\big\{\xi^{-1}\operatorname{id}-(\phi_L[-1]\widetilde T_f)^0_x\big\}= \dim H^0\big(\phi_L[-1] \Idot_\U(\xi)\big)_x\geq 0.
$$

\smallskip

In particular, if $x$ is an isolated singular point of $X$, then
$$
\big(\Gamma_{f, L}^1\cdot V(L)\big)_x-\dim \operatorname{ker}\big\{\xi^{-1}\operatorname{id}-(\widetilde T_f)^0_x\big\}= \dim H^0\big(\phi_L[-1] \Idot_\U(\xi)\big)_x\geq 0.
$$
\end{cor}

\medskip

\begin{rem} If $x$ is an isolated singular point of $X$, then \corref{cor:xi} and \corref{cor:vanvan} combine to tell us the curious fact that $\big(\Gamma_{f, L}^1\cdot V(L)\big)_x$ is an upper-bound for the dimension/rank of all of the eigenspaces of the Milnor monodromy in the one non-trivial degree, degree $n$.
\end{rem}

\bigskip

As a final application of \thmref{thm:1theorem}, we now consider the setting of \cite{heplermassparam} and  \cite{hepler}, in which $X=V(f)$ has a smooth normalization $M$. In this case, the normalization map $F:M\rightarrow X$ is a parameterization of $X$, a finite, surjective, analytic map which is an analytic isomorphism over $X\backslash\Sigma$. This necessarily requires that $\Sigma$ is purely of codimension 1 inside of $X$ (this vacuously includes the case where $\Sigma=\emptyset$). It follows from the support and cosupport characterization of intersection cohomology that $\Idot_X\cong F_*\Z^\bullet_M[n]$. For all $x\in X$, we let $m(x):=|F^{-1}(x)|$ be the number of points in the fiber over $x$ ({\bf not} counted with any sort of algebraic multiplicity).

Via this isomorphism, on the stalk cohomology at a point $x\in X$, the canonical map $\tau_{{}_X}:\Z^\bullet_X[n]\rightarrow\Idot_X$ induces the diagonal map in the only non-zero degree:
$$
\tau_{{}_{X, x}}: H^{-n}(\Z^\bullet_X[n])_x\cong \Z\rightarrow \Z^{m(x)}\cong \bigoplus_{y\in F^{-1}(x)}H^{-n}(\Z^\bullet_M[n])_y.
$$
It follows that the stalk cohomology of $\Ndot_X$ is given by
$$
H^k(\Ndot_X)_x\cong\begin{cases} \Z^{m(x)-1}, &\textnormal{ if } k=-n+1;\\
0, &\textnormal{ if } k\neq -n+1.
\end{cases}
$$
In this context, we have defined $\Ndot_X$ to be the {\it multiple-point complex} of $F$.

\bigskip

Using our results and notation above, we quickly conclude:

\begin{cor} Suppose that $X=V(f)$ is $2$-dimensional with a smooth normalization, and let $x\in\Sigma$, so that $\dim_x\Sigma=1$.

Let
$$
 H^{1}(F_{f, x};\,\Z)\xrightarrow{(T_f)^{-1}_x} H^{1}(F_{f, x};\,\Z)
$$
be the Milnor monodromy on the cohomology of the Milnor fiber of $f$ at $x$.

Then, $m(x)=\operatorname{rank} H^{3}(K_{X, x};\,\Z)$, which equals the number of irreducible components of $X$ at $x$, and there are isomorphisms
$$
\operatorname{ker}\big\{\operatorname{id}- (\widetilde T_f)^{-1}_x\big\}\cong \Z^\omega\cong \widetilde{IH}^{0}(B_\epsilon^\circ(x)\cap X; \, \Z),
$$
where $0<\epsilon\ll 1$ and $\omega=m(x)-1$.

Furthermore, suppose that $\dim_x\Sigma(f_{|_{V(L)}})=0$. Then, there is an equality
$$
\operatorname{rank}\operatorname{ker}\big\{\operatorname{id}-\widehat T_{f, L}\big\}=\big(\Gamma_{f, L}^1\cdot V(L)\big)_x-\operatorname{rank}\operatorname{coker}\{r_{{}_{X, x}}\}= -m(x)+\sum_im(x_i),
$$
where $\widehat T_{f, L}$ is the automorphism induced by the Milnor monodromy of $f$ on the relative cohomology
$$
H^2\big(F_{f, x}, \,\dot{\cup_i} F_{f, x_i}; \,\Z\big),
$$
 the summation and union are over $x_i\in B^\circ_\epsilon(x)\cap \Sigma\cap L^{-1}(\gamma)$,  $0< |\gamma|\ll\epsilon\ll 1$, and $F_{f, x}$ and $F_{f, x_i}$ denote the Milnor fibers of $f$ at the respective points and, finally, 
$$
r_{{}_{X, x}}: {IH}^{1}(B^\circ_\epsilon(x)\cap X;\, \Z)\rightarrow{IH}^{1}(B^\circ_\epsilon(x)\cap X\cap L^{-1}(\gamma);\, \Z),
$$
is induced by restriction.

\end{cor}

\end{document}